\documentclass[12pt,reqno]{article}

\usepackage{amsmath,amssymb,amsfonts,amsthm}
\usepackage{mathtools}
\usepackage{bm}
\usepackage[margin=1.25in]{geometry}

\usepackage{hyperref}

\usepackage{hyperref}
\usepackage[nameinlink,capitalise]{cleveref}
\usepackage{orcidlink}

\newcommand{\doi}[1]{\href{https://doi.org/#1}{doi: #1}}

\newtheorem{theorem}{Theorem}[section]
\newtheorem{proposition}[theorem]{Proposition}
\newtheorem{lemma}[theorem]{Lemma}
\newtheorem{corollary}[theorem]{Corollary}

\theoremstyle{definition}
\newtheorem{definition}[theorem]{Definition}

\theoremstyle{remark}
\newtheorem{remark}[theorem]{Remark}

\newcommand{\R}{\mathbb{R}}
\newcommand{\cX}{\mathcal{X}}
\newcommand{\cF}{\mathcal{F}}
\newcommand{\cG}{\mathcal{G}}
\newcommand{\cR}{\mathcal{R}}
\newcommand{\cO}{\mathcal{O}}
\newcommand{\D}{\mathrm{D}}

\title{Higher Order Structure along Nilpotent Eigendirections
in Planar Vector Fields}

\author{
Roberto Albarran-Garc\'ia\orcidlink{0009-0006-4591-7316}\\
\small Departamento de Matem\'aticas, UAM--Iztapalapa,\\
\small 09310 Iztapalapa, Mexico City, Mexico\\
\small \texttt{albarrangr74@live.com.mx}
\and
Martha Alvarez-Ram\'irez\orcidlink{0000-0001-9187-1757}\\
\small Departamento de Matem\'aticas, UAM--Iztapalapa,\\
\small 09310 Iztapalapa, Mexico City, Mexico\\
\small \texttt{mar@xanum.uam.mx}
\and
Marco Polo Garc\'ia-Rivera\orcidlink{0009-0005-4774-3383}\\
\small Facultad de Ingenier\'ia, Universidad La Salle,\\
\small Mexico City, Mexico\\
\small \texttt{mpgr86@gmail.com}
}

\date{}
\begin{document}

\maketitle

\begin{abstract}
We study a distinguished one dimensional restriction associated with
planar vector fields whose linearization is a nonzero nilpotent matrix
of index two. This restriction admits an intrinsic interpretation on
the nilpotent eigendirection, while its scalar representation in Jordan
coordinates satisfies a simple transformation law under admissible
changes of Jordan chain. Its quadratic Taylor coefficient coincides
with the intrinsic quadratic Bogdanov--Takens coefficient.
For rational restrictions, we show that the Taylor coefficient
sequence satisfies a finite linear recurrence determined by the
minimal denominator. Its degree defines a recurrence degree that is
invariant under admissible changes of Jordan chain and gives the
minimal order of an eventual homogeneous recurrence. Exact
cancellations may therefore be detected directly from the restriction
before nonlinear normal form transformations are performed.
Applications to four planar models exhibit different recurrence degree
and cancellation mechanisms, illustrating the complementarity between
the distinguished restriction and smooth nilpotent normal form theory.
\end{abstract}

\newpage


\medskip


\section{Introduction}
\label{sec:introduction}

Nilpotent singularities constitute one of the fundamental local
configurations in the theory of planar vector fields. In the simplest
nontrivial case, the linearization at an equilibrium is a nonzero
nilpotent matrix of index two. After introducing coordinates adapted
to a Jordan chain, the vector field can be written as
\begin{equation}
\label{eq:intro-nilpotent}
\dot u=v+\mathcal F(u,v),
\qquad
\dot v=\mathcal G(u,v),
\end{equation}
where
\[
\mathcal F(u,v),\mathcal G(u,v)
=
\mathcal O\bigl(\|(u,v)\|^2\bigr).
\]
This setting underlies the local analysis of Bogdanov--Takens
singularities and their higher-codimension degeneracies; see, among
others, \cite{Takens1974,Bogdanov1981,GuckenheimerHolmes1983,
Kuznetsov2004,Kuznetsov2005}. Related classes of planar nilpotent systems under additional
structural assumptions, including reversibility and symmetry, have
also been investigated; see, for instance,
\cite{CorberaValls2024x,CorberaValls2024y}.

The classical normal form approach proceeds by successive
near--identity transformations, and possibly time reparametrizations,
that remove nonessential terms and isolate the coefficients relevant
to the local classification. These procedures are particularly
effective for determining the codimension of degenerate nilpotent
singularities, but in doing so they combine information already
present in the linearly reduced vector field with contributions
generated along the way, during nonlinear normalization.

Recently, Castellanos and Chan--L\'opez
\cite{CastellanosChanLopez2026} obtained an intrinsic
characterization of the two quadratic Bogdanov--Takens normal form
coefficients in the planar case in terms of directional derivatives
of the determinant and trace of the Jacobian along the nilpotent
eigendirection. Chan--L\'opez subsequently extended this approach to
arbitrary dimension~\cite{ChanLopez2026Rn}, expressing the quadratic
coefficients through directional derivatives of characteristic
invariants. These results provide a direct description of the
quadratic Bogdanov--Takens data from the variation of the
linearization along the kernel direction.

Here we take a complementary route. We study
higher order information that is already present along the nilpotent
eigendirection before nonlinear normal form transformations combine
pure and mixed terms. For a planar system of the
form~\eqref{eq:intro-nilpotent}, we consider the scalar restriction
\begin{equation}
\label{eq:intro-restriction}
\mathcal R_{\mathcal X}(u):=\mathcal G(u,0).
\end{equation}
Its Taylor coefficients are precisely the coefficients associated
with the pure \(u\)-monomials in the second component of the linearly
reduced vector field. The quadratic coefficient coincides with the
intrinsic Bogdanov--Takens coefficient usually denoted by \(a\),
whereas the coefficients of degree three and higher retain additional
information about the restriction before nonlinear normalization.

The scalar function~\eqref{eq:intro-restriction} admits a useful
intrinsic interpretation. If \(A=D\mathcal X(0)\) and
\(K=\ker A=\operatorname{im}A\), the local map
\[
\rho_{\mathcal X}(\xi):=A\mathcal X(\xi),
\qquad \xi\in K
\]
sufficiently close to the origin, takes values in \(K\) and is defined
without choosing a Jordan chain. After a generator of \(K\) is fixed,
\(\mathcal R_{\mathcal X}\) is the corresponding scalar
representation of this map, keeping the
underlying one-dimensional object separate from the coordinates used to
represent it.

The higher order information considered here is therefore understood
at the level of the linearly reduced vector field. We study the Taylor
coefficient sequence of \(\mathcal R_{\mathcal X}\) and the algebraic
relations among its terms, modulo admissible changes of Jordan chain,
with no claim that the coefficients of degree three and higher remain
invariant under arbitrary nonlinear changes of phase coordinates.
Keeping these two notions of equivalence apart is what makes it
possible, later on, to compare the distinguished restriction
meaningfully with a smooth nilpotent normal form.

The viewpoint developed here was initially suggested by an exact
rational identity arising in a modified Leslie--Gower model, where the
complete Taylor coefficient sequence of the distinguished restriction
can be obtained without an order-by-order expansion
\cite{AlbarranAlvarezGarcia2026}. The present work develops the
underlying mechanism for general planar nilpotent vector fields.

We first determine how the scalar representation changes under an
admissible change of Jordan chain. If two chains are related by the
usual rescaling of the kernel direction, their restrictions satisfy
\[
\widetilde{\mathcal R}_{\mathcal X}(\widetilde u)
=
\frac{1}{a}
\mathcal R_{\mathcal X}(a\widetilde u),
\qquad a\neq0.
\]
So although the individual Taylor coefficients depend on the
normalization of the generator of \(K\), their vanishing pattern and
the order of the restriction do not: both are independent of the
chosen Jordan chain.

For rational distinguished restrictions,
\[
\mathcal R_{\mathcal X}(u)
=
u^2\frac{P(u)}{Q(u)},
\qquad Q(0)\neq0,
\]
the Taylor coefficient sequence satisfies a finite linear recurrence
determined by the denominator \(Q\). After common factors of \(P\)
and \(Q\) are cancelled, the degree of the minimal denominator defines
the recurrence degree of the restriction. We prove that this degree is
independent of the admissible Jordan chain and is the minimal order of
an eventual homogeneous linear recurrence with constant coefficients
satisfied by the Taylor coefficient sequence. Thus a finite algebraic
object controls the recurrence structure of an infinite Taylor
sequence.

The case of a linear denominator admits a particularly explicit
description. If
\[
\mathcal R_{\mathcal X}(u)
=
Cu^2\frac{a+bu}{d-eu},
\]
then every Taylor coefficient of degree at least three contains the
common factor $ae+bd.$
Its vanishing produces an exact numerator--denominator cancellation
and makes the distinguished restriction quadratic. More generally,
cancellations in the reduced rational expression are reflected in
changes of the recurrence degree.

These statements concern the distinguished restriction, not the full
nonlinear nilpotent normal form. Mixed monomials disappear when
\(v=0\) is imposed, but they may contribute to normal form
coefficients through the homological equations associated with
nonlinear near--identity transformations, so exact recurrences or
all-orders cancellations in \(\mathcal R_{\mathcal X}\) need not
correspond to degeneracies of the same order in a smooth nilpotent
normal form. The two constructions retain different, complementary
parts of the local information.

The general results are illustrated by four planar models. In the
mosquito population model~\cite{HuangRuanYuZhang2019}, the generic
recurrence degree is two and drops to one on an explicit algebraic
locus where the denominator factors coalesce and a common factor
cancels. In a modified Leslie--Gower model~\cite{ZhaoZhao2026}, the
recurrence degree drops from one to zero when an exact cancellation
makes the distinguished restriction quadratic. In a second modified
Leslie--Gower model~\cite{WangWuZou2026}, the recurrence degree remains
one throughout the admissible parameter region because the
corresponding cancellation factor is strictly positive. Finally, for
the Bazykin model considered in~\cite{CastellanosChanLopez2026}, the
recurrence degree remains two even at the parameter value where the
second intrinsic quadratic Bogdanov--Takens coefficient vanishes and
the two denominator factors coalesce. The four examples therefore
exhibit distinct relations between denominator structure,
cancellation, and normal form degeneracy.

The paper is organized as follows.
Section~\ref{sec:nilpotent} introduces the intrinsic map on the
nilpotent eigendirection and its scalar distinguished restriction.
Section~\ref{sec:invariance} determines the transformation law under
changes of Jordan chain and identifies the resulting invariant
properties. Section~\ref{sec:rational} develops the theory of rational
restrictions, minimal denominators, and finite recurrence relations.
Section~\ref{sec:normalforms} clarifies the relation between the
distinguished restriction and smooth nilpotent normal forms.
Section~\ref{sec:applications} applies the theory to four planar
models and compares the resulting recurrence mechanisms.

\section{Nilpotent singularities and distinguished restrictions}
\label{sec:nilpotent}

We begin by introducing the distinguished restriction for a planar
vector field with a nilpotent singularity. The construction is first
formulated directly on the one-dimensional kernel of the linearization
and is then related to its scalar representation in Jordan coordinates.
This separates the underlying geometric object from the choice of a
Jordan chain and identifies the Taylor coefficients that will be
studied in the subsequent sections.

Consider
\begin{equation}
\label{eq:general-system}
\dot z=\cX(z),
\qquad z\in\R^2,
\end{equation}
where \(\cX\) is a \(C^r\) vector field, \(r\geq2\), defined in a
neighborhood of the origin. We assume throughout this section that
$\cX(0)=0$ and that
\[
A:=\D\cX(0)
\]
is a nonzero nilpotent matrix of index two:
\begin{equation}
\label{eq:nilpotent-assumption}
A\neq0,
\qquad
A^2=0.
\end{equation}

Since \(A\) is a nonzero nilpotent \(2\times2\) matrix,
\[
\operatorname{rank}A=1,
\qquad
\operatorname{im}A=\ker A.
\]
Set
\[
K:=\ker A.
\]
Thus \(K\) is the one-dimensional eigendirection associated with the
zero eigenvalue.

The following interpretation of the distinguished restriction was
suggested to us by  Chan--L\'opez. Consider the map
\begin{equation}
\label{eq:intrinsic-restriction}
\rho_{\cX}:K\longrightarrow K,
\qquad
\rho_{\cX}(\xi):=A\cX(\xi).
\end{equation}
This map is well defined because
\(A\cX(\xi)\in\operatorname{im}A=K\) for every \(\xi\in K\)
sufficiently close to the origin. In particular, \(\rho_{\cX}\) is
defined directly on the distinguished eigendirection and does not
require the choice of a generalized eigenvector.

\begin{definition}
\label{def:distinguished-restriction}
Choose \(q_0\neq0\) spanning \(K\). Since
\(K=\operatorname{span}\{q_0\}\), there is a unique scalar function
\(\cR_{\cX}\), defined near the origin, such that
\begin{equation}
\label{eq:scalar-representation}
\rho_{\cX}(u q_0)
=
\cR_{\cX}(u)q_0.
\end{equation}
We call \(\cR_{\cX}\) the \emph{distinguished eigendirection
restriction} of \(\cX\) associated with the choice of \(q_0\).
\end{definition}

To recover its expression in Jordan coordinates, choose \(q_1\) such
that
\begin{equation}
\label{eq:jordan-chain}
Aq_0=0,
\qquad
Aq_1=q_0,
\end{equation}
and set $P=(q_0\ q_1).$ Introduce Jordan coordinates through
\begin{equation}
\label{eq:jordan-coordinates}
z=P
\begin{pmatrix}
u\\ v
\end{pmatrix}
=
uq_0+vq_1.
\end{equation}
Since
\[
P^{-1}AP=
\begin{pmatrix}
0&1\\
0&0
\end{pmatrix},
\]
system~\eqref{eq:general-system} becomes
\begin{equation}
\label{eq:nilpotent-system}
\dot u=v+\cF(u,v),
\qquad
\dot v=\cG(u,v),
\end{equation}
with
\[
\cF(u,v),\cG(u,v)
=
\cO\bigl(\|(u,v)\|^2\bigr).
\]

\begin{proposition}
\label{prop:scalar-jordan}
The scalar representation~\eqref{eq:scalar-representation} is given
in the Jordan coordinates~\eqref{eq:jordan-coordinates} by
\begin{equation}
\label{eq:distinguished-restriction}
\cR_{\cX}(u)=\cG(u,0).
\end{equation}
In particular, once \(q_0\) is fixed, \(\cR_{\cX}\) is independent
of the choice of the generalized eigenvector \(q_1\).
\end{proposition}

\begin{proof}
Along \(K\), we have \(z=uq_0\). From
\eqref{eq:nilpotent-system},
\[
\cX(uq_0)
=
q_0\cF(u,0)+q_1\cG(u,0).
\]
Applying \(A\) and using
\[
Aq_0=0,
\qquad
Aq_1=q_0,
\]
gives
\[
A\cX(uq_0)
=
\cG(u,0)q_0.
\]
Comparison with~\eqref{eq:scalar-representation} yields
\[
\cR_{\cX}(u)=\cG(u,0).
\]
\end{proof}

\begin{proposition}
\label{prop:pure-coefficients}
Let \(\cX\) satisfy~\eqref{eq:nilpotent-assumption}. Then
\[
\cR_{\cX}(0)=0,
\qquad
\cR_{\cX}'(0)=0.
\]
Consequently,
\begin{equation}
\label{eq:R-expansion}
\cR_{\cX}(u)
=
\sum_{k=2}^{r}g_k u^k+o(u^r).
\end{equation}
The coefficients \(g_k\) will be referred to as the Taylor
coefficients of the distinguished restriction. In Jordan coordinates,
\(g_k\) is precisely the coefficient of the pure monomial \(u^k\)
in the second component of the vector field.
\end{proposition}

\begin{proof}
Since the origin is an equilibrium of
\eqref{eq:nilpotent-system}, $\cG(0,0)=0.$
Furthermore, the linear part of the second equation of
\eqref{eq:nilpotent-system} vanishes identically because the
linearization has canonical nilpotent form
\[
\begin{pmatrix}
0&1\\
0&0
\end{pmatrix}.
\]
Hence $D\cG(0,0)=0,$
and, by~\eqref{eq:distinguished-restriction},
\[
\cR_{\cX}(0)=0,
\qquad
\cR_{\cX}'(0)=0.
\]
The expansion~\eqref{eq:R-expansion} follows from Taylor's theorem.
Since \(\cR_{\cX}(u)=\cG(u,0)\), setting \(v=0\) eliminates every
monomial containing a positive power of \(v\), leaving precisely the
pure monomials \(u^k\) in the second component.
\end{proof}

\begin{remark}
\label{rem:not-normal form}
The Taylor coefficients \(g_k\) in~\eqref{eq:R-expansion} are those
of the distinguished restriction. Equivalently, in Jordan coordinates
they are the coefficients of the pure monomials \(u^k\) in the second
component of the linearly reduced vector field. They are not
coefficients of a nonlinear normal form. In general, the latter also
contain contributions generated by mixed terms during successive
near--identity transformations, a point we return to below.
\end{remark}

\begin{proposition}
\label{prop:g2-intrinsic}
Let \(\cX\in C^2\) satisfy~\eqref{eq:nilpotent-assumption}, and let
\(q_0\) be the generator of \(K\) used in
\eqref{eq:scalar-representation}. Then the quadratic Taylor
coefficient of the distinguished restriction satisfies
\begin{equation}
\label{eq:g2-intrinsic}
g_2=
-\frac12\,D_{q_0}\bigl(\det D\cX\bigr)(0).
\end{equation}
Consequently, \(g_2\) coincides with the quadratic Bogdanov--Takens
normal form coefficient usually denoted by \(a\).
\end{proposition}

\begin{proof}
In the Jordan coordinates~\eqref{eq:nilpotent-system}, write
\[
Y(u,v)=
\begin{pmatrix}
v+\cF(u,v)\\
\cG(u,v)
\end{pmatrix}.
\]
Since \(\cF,\cG=\cO(\|(u,v)\|^2)\),
\[
DY(u,v)=
\begin{pmatrix}
\cF_u & 1+\cF_v\\
\cG_u & \cG_v
\end{pmatrix},
\]
and therefore
\[
\det DY
=
\cF_u\cG_v-(1+\cF_v)\cG_u.
\]
Differentiating with respect to \(u\) and evaluating at the origin,
all terms vanish except the derivative of \(-\cG_u\), so that
\[
\partial_u\det DY(0,0)
=
-\cG_{uu}(0,0)
=
-2g_2.
\]
Because the \(u\)-axis is generated by \(q_0\), this gives
\eqref{eq:g2-intrinsic}. The identification with the quadratic
Bogdanov--Takens coefficient follows from the intrinsic formula of
Castellanos and Chan--L\'opez~\cite{CastellanosChanLopez2026}.
\end{proof}

\begin{remark}
\label{rem:beyond-intrinsic-g2}
Castellanos and Chan--L\'opez~\cite{CastellanosChanLopez2026} also show
that the second quadratic Bogdanov--Takens coefficient is
\[
b=D_{q_0}\bigl(\operatorname{tr}D\cX\bigr)(0).
\]
In the Jordan coordinates~\eqref{eq:nilpotent-system}, this becomes
\[
b=\cF_{uu}(0,0)+\cG_{uv}(0,0),
\]
so it depends on mixed information that is not contained in
\(\cR_{\cX}(u)=\cG(u,0)\) alone. Thus the distinguished restriction
recovers the intrinsic quadratic coefficient \(a=g_2\), while its
additional content lies in the higher Taylor coefficients
\(g_3,g_4,\ldots\), whose organization is studied below before any
nonlinear normal form transformation is performed.
\end{remark}

\section{Changes of Jordan chain and structural invariance}
\label{sec:invariance}
The intrinsic map \(\rho_{\cX}:K\to K\) introduced in
Section~\ref{sec:nilpotent} is independent of any Jordan chain,
whereas its scalar representation
\(\cR_{\cX}\) depends on the choice of a generator of
\(K=\ker A\). We now determine precisely how this scalar
representation changes under a different admissible choice of
Jordan chain. This will allow us to identify which properties of the
restriction are independent of that choice.

Let \(A\neq0\) be a nilpotent \(2\times2\) matrix satisfying
\(A^2=0\), and let
\[
Aq_0=0,
\qquad
Aq_1=q_0
\]
be a Jordan chain.

\begin{lemma}
\label{lem:jordan-chains}
Every other Jordan chain
\((\widetilde q_0,\widetilde q_1)\) satisfying
\[
A\widetilde q_0=0,
\qquad
A\widetilde q_1=\widetilde q_0
\]
is of the form
\begin{equation}
\label{eq:jordan-chain-change}
\widetilde q_0=a q_0,
\qquad
\widetilde q_1=a q_1+bq_0,
\end{equation}
for some \(a\in\R\setminus\{0\}\) and \(b\in\R\).
\end{lemma}

\begin{proof}
Since \(A\neq0\), \(A^2=0\), and \(A\) acts on \(\R^2\), we have
\[
\dim\ker A=1.
\]
Hence \(\widetilde q_0\in\ker A=\operatorname{span}\{q_0\}\), so
$\widetilde q_0=a q_0$ 
for some \(a\neq0\). Moreover,
\[
A(\widetilde q_1-aq_1)
=
A\widetilde q_1-aAq_1
=
a q_0-a q_0
=0.
\]
Thus
\[
\widetilde q_1-aq_1\in\ker A
=
\operatorname{span}\{q_0\},
\]
and therefore $\widetilde q_1=a q_1+bq_0$, 
for some \(b\in\R\).
\end{proof}

The preceding lemma shows that changing the Jordan chain involves two
independent freedoms: a rescaling of the generator \(q_0\) by
\(a\neq0\), and the addition of a multiple of \(q_0\) to the
generalized eigenvector \(q_1\). Since the distinguished restriction
is the scalar representation of \(\rho_{\cX}\) relative to \(q_0\),
only the first of these affects that representation.

\begin{theorem}
\label{thm:restriction-transformation}
Let \(\cR_{\cX}\) and \(\widetilde{\cR}_{\cX}\) denote the
distinguished eigendirection restrictions associated with the Jordan
chains \((q_0,q_1)\) and
\((\widetilde q_0,\widetilde q_1)\), respectively. If the two chains
are related by~\eqref{eq:jordan-chain-change}, then
\begin{equation}
\label{eq:R-transformation}
\widetilde{\cR}_{\cX}(\widetilde u)
=
\frac{1}{a}\,
\cR_{\cX}(a\widetilde u).
\end{equation}
In particular, the transformation law is independent of the parameter
\(b\).
\end{theorem}

\begin{proof}
By Definition~\ref{def:distinguished-restriction},
\[
\rho_{\cX}(u q_0)
=
\cR_{\cX}(u)q_0.
\]
For the second Jordan chain,
$\widetilde q_0=a q_0,$
and hence
\[
\rho_{\cX}(\widetilde u\,\widetilde q_0)
=
\rho_{\cX}(a\widetilde u\,q_0)
=
\cR_{\cX}(a\widetilde u)q_0.
\]
Since \(q_0=a^{-1}\widetilde q_0\), this becomes
\[
\rho_{\cX}(\widetilde u\,\widetilde q_0)
=
\frac{1}{a}\,
\cR_{\cX}(a\widetilde u)\widetilde q_0.
\]
On the other hand, the scalar representation associated with
\(\widetilde q_0\) is defined by
\[
\rho_{\cX}(\widetilde u\,\widetilde q_0)
=
\widetilde{\cR}_{\cX}(\widetilde u)\widetilde q_0.
\]
Comparison of the two expressions gives
\[
\widetilde{\cR}_{\cX}(\widetilde u)
=
\frac{1}{a}\,
\cR_{\cX}(a\widetilde u).
\]
The parameter \(b\) does not enter because it changes only the
generalized eigenvector \(q_1\), whereas the scalar representation
of \(\rho_{\cX}\) depends only on the generator of \(K\).
\end{proof}

\begin{corollary}
\label{cor:coefficient-transformation}
Let \(\cX\in C^r\), \(r\ge2\), and write
\[
\cR_{\cX}(u)
=
\sum_{k=2}^{r}g_k u^k+o(u^r),
\qquad
\widetilde{\cR}_{\cX}(\widetilde u)
=
\sum_{k=2}^{r}\widetilde g_k\widetilde u^k
+o(\widetilde u^r).
\]
Then the Taylor coefficients of the two distinguished restrictions
satisfy
\begin{equation}
\label{eq:gk-transformation}
\widetilde g_k=a^{k-1}g_k,
\qquad 2\le k\le r.
\end{equation}
Consequently, within the available jet order, the vanishing or
nonvanishing of each Taylor coefficient is independent of the chosen
Jordan chain. If \(\cR_{\cX}\) is analytic, in particular rational
near the origin, the same formula holds for every \(k\ge2\).
\end{corollary}

\begin{proof}
Substituting the Taylor expansion of \(\cR_{\cX}\) into
\eqref{eq:R-transformation} gives
\[
\widetilde{\cR}_{\cX}(\widetilde u)
=
\frac1a
\left(
\sum_{k=2}^{r}g_k(a\widetilde u)^k
+o(\widetilde u^r)
\right)
=
\sum_{k=2}^{r}
a^{k-1}g_k\widetilde u^k
+o(\widetilde u^r).
\]
Comparison of Taylor coefficients yields
\eqref{eq:gk-transformation}. In the analytic case the same argument
applies to the full convergent Taylor series.
\end{proof}

\begin{remark}
\label{rem:scope-invariance}
The preceding statements concern changes of Jordan chain for a fixed
vector field in a fixed phase-space chart. They do not assert that the
higher Taylor coefficients \(g_k\), \(k\ge3\), are invariants under
arbitrary nonlinear changes of phase coordinates. Rather, the
structural information considered here is attached to the
distinguished restriction modulo admissible changes of Jordan chain.
Keeping the two notions apart matters when comparing the restriction with
smooth nilpotent normal forms, where nonlinear near--identity
transformations may mix pure and mixed terms.
\end{remark}

\begin{definition}
\label{def:restriction-order}
Assume that \(\cR_{\cX}\not\equiv0\). We define the
\emph{restriction order} by
\begin{equation}
\label{eq:restriction-order}
\nu(\cX):=\operatorname{ord}_{u=0}\cR_{\cX}
=\min\{k\ge2:\ g_k\neq0\}.
\end{equation}
If \(\cR_{\cX}\equiv0\), we set
$\nu(\cX)=\infty.$
\end{definition}

\begin{corollary}
\label{cor:restriction-order-invariant}
The restriction order \(\nu(\cX)\) is independent of the Jordan
chain used to define \(\cR_{\cX}\).
\end{corollary}

\begin{proof}
By Corollary~\ref{cor:coefficient-transformation},
\[
\widetilde g_k=a^{k-1}g_k,
\qquad a\neq0.
\]
Hence
\[
\widetilde g_k=0
\quad\Longleftrightarrow\quad
g_k=0
\]
for every \(k\ge2\). Therefore the index of the first nonzero Taylor
coefficient is unchanged.
\end{proof}


\section{Rational restrictions and recurrence relations}
\label{sec:rational}
We now consider the case in which the distinguished restriction is
rational. This situation arises naturally in planar models whose
vector fields contain rational nonlinearities, but the results of this
section are purely algebraic and do not depend on any particular
dynamical system. The main point is that the denominator of the exact
restriction imposes a finite recurrence on its Taylor coefficients.

Suppose that
\begin{equation}
\label{eq:rational-restriction-general}
\cR_{\cX}(u)
=
u^2\frac{P(u)}{Q(u)},
\qquad
Q(0)\neq0,
\end{equation}
where
\[
P(u)=\sum_{j=0}^{m}p_j u^j,
\qquad
Q(u)=\sum_{j=0}^{d}q_j u^j,
\qquad
q_0\neq0.
\]
Since \(Q(0)\neq0\), the quotient \(P/Q\) is analytic in a
neighborhood of the origin. We therefore write
\begin{equation}
\label{eq:rational-series}
\cR_{\cX}(u)
=
u^2\sum_{n\ge0}c_nu^n
=\sum_{k\ge2}g_k u^k,
\qquad
g_{n+2}=c_n.
\end{equation}

\begin{theorem} 
\label{thm:rational-recurrence}
Suppose that the distinguished restriction has the rational
representation~\eqref{eq:rational-restriction-general}. Then its
Taylor coefficients are completely determined by
\begin{equation}
\label{eq:inhomogeneous-recurrence}
q_0c_n+q_1c_{n-1}+\cdots+q_dc_{n-d}=p_n,
\qquad n\ge0,
\end{equation}
where
\[
c_j=0\quad\text{for }j<0,
\qquad
p_n=0\quad\text{for }n>m.
\]
Equivalently,
\begin{equation}
\label{eq:recursive-cn}
c_n
=
\frac{1}{q_0}
\left(
p_n-\sum_{j=1}^{d}q_jc_{n-j}
\right).
\end{equation}
In particular, for every \(n>m\), the sequence satisfies the
homogeneous recurrence
\begin{equation}
\label{eq:homogeneous-recurrence}
q_0c_n+q_1c_{n-1}+\cdots+q_dc_{n-d}=0.
\end{equation}
Thus the complete sequence of Taylor coefficients of the
distinguished restriction is determined by finitely many coefficients
of \(P\) and \(Q\).
\end{theorem}

\begin{proof}
From~\eqref{eq:rational-series},
\[
\frac{P(u)}{Q(u)}
=
\sum_{n\ge0}c_nu^n.
\]
Multiplying by \(Q(u)\) gives
\[
P(u)
=
\left(
\sum_{j=0}^{d}q_ju^j
\right)
\left(
\sum_{n\ge0}c_nu^n
\right).
\]
Comparison of the coefficient of \(u^n\) yields
\[
p_n
=
\sum_{j=0}^{d}q_jc_{n-j},
\]
with the convention \(c_j=0\) for \(j<0\). Since \(q_0\neq0\),
solving for \(c_n\) gives~\eqref{eq:recursive-cn}. Finally,
\(p_n=0\) for \(n>m\), and hence
\eqref{eq:homogeneous-recurrence} follows.
\end{proof}

\begin{remark}
\label{rem:denominator-controls}
Theorem~\ref{thm:rational-recurrence} shows that the denominator of
the exact restriction controls the higher order structure of its
Taylor coefficients. Once the finitely many coefficients required to
initialize the recurrence are known, no further Taylor expansion is
needed to determine the remaining coefficients.
\end{remark}

\begin{corollary}
\label{cor:linear-denominator}
Suppose that
\begin{equation}
\label{eq:linear-denominator}
\cR_{\cX}(u)
=
C u^2\frac{a+bu}{d-eu},
\qquad
Cd\neq0.
\end{equation}
Then
\begin{equation}
\label{eq:g2-linear}
g_2=\frac{Ca}{d},
\end{equation}
whereas, for every \(k\ge3\),
\begin{equation}
\label{eq:gk-linear}
g_k=C\frac{e^{\,k-3}}{d^{\,k-1}}\bigl(ae+bd\bigr).
\end{equation}
As a result, all Taylor coefficients of degree at least three share
the factor
\begin{equation}
\label{eq:Sigma}
\Sigma:=ae+bd.
\end{equation}
\end{corollary}

\begin{proof}
Using
\[
\frac{1}{d-eu}
=
\frac1d
\sum_{j\ge0}
\left(\frac{e}{d}u\right)^j,
\]
we obtain
\[
\cR_{\cX}(u)
=
\frac{C}{d}u^2(a+bu)
\sum_{j\ge0}
\left(\frac{e}{d}u\right)^j.
\]
The coefficient of \(u^2\) is \(Ca/d\). For \(k\ge3\), the
coefficient of \(u^k\) receives one contribution from \(a\) and one
from \(bu\), giving
\[
g_k
=
C\left(
\frac{ae^{k-2}}{d^{k-1}}
+
\frac{be^{k-3}}{d^{k-2}}
\right)
=
C\frac{e^{k-3}}{d^{k-1}}(ae+bd).
\]
\end{proof}

\begin{corollary}
\label{cor:all-orders-cancellation}
Under the hypotheses of Corollary~\ref{cor:linear-denominator}, the
following statements are equivalent:
\begin{enumerate}
\item[{1.}] \(\Sigma=ae+bd=0\);
\item[{2.}] \(g_k=0\) for every \(k\ge3\);
\item[{3.}] the distinguished restriction is exactly quadratic,
\[
\cR_{\cX}(u)=\frac{Ca}{d}u^2.
\]
\end{enumerate}
\end{corollary}

\begin{proof}
The equivalence between 1.  and 2. follows immediately from
\eqref{eq:gk-linear}. If \(\Sigma=0\), then
$ae+bd=0,$ and hence
\[a+bu=\frac{a}{d}(d-eu).\]
Substitution into~\eqref{eq:linear-denominator} gives
\[
\cR_{\cX}(u)=\frac{Ca}{d}u^2.
\]
Conversely, if the restriction is exactly quadratic, then all its
Taylor coefficients of degree at least three vanish, and (2) implies
(1).
\end{proof}

\begin{remark}
\label{rem:cancellation-invariant}
Although the quantities entering
\eqref{eq:linear-denominator} depend on the chosen representation,
the condition
\[
g_k=0
\qquad\text{for all }k\ge3
\]
is independent of the admissible Jordan chain. Indeed, by
Corollary~\ref{cor:coefficient-transformation}, simultaneous
vanishing of these Taylor coefficients is preserved under every
admissible change of chain. Therefore, for a representation of the
form~\eqref{eq:linear-denominator}, the condition \(\Sigma=0\)
characterizes a structural cancellation of the distinguished
restriction rather than a cancellation of a single finite-order
coefficient.
\end{remark}

\begin{definition}
\label{def:minimal-denominator}
Suppose that the distinguished restriction is rational and not
identically zero. Write
\[
\cR_{\cX}(u)
=
u^2\frac{P(u)}{Q(u)},
\qquad
\gcd(P,Q)=1,
\qquad
Q(0)\neq0.
\]
The polynomial \(Q\), determined up to multiplication by a nonzero
constant, will be called the \emph{minimal denominator} of the
distinguished restriction. Its degree
\[
d_{\cX}:=\deg Q
\]
will be called the \emph{recurrence degree} of the restriction. If
\(\cR_{\cX}/u^2\) is a polynomial, then \(Q=1\) and
\(d_{\cX}=0\). For the identically zero restriction we also set
\(d_{\cX}=0\).
\end{definition}

\begin{proposition}
\label{prop:recurrence-degree-invariant}
Let \(\cR_{\cX}\) be rational. Then the recurrence degree
\(d_{\cX}\) is independent of the admissible Jordan chain used to
define the distinguished restriction.
\end{proposition}

\begin{proof}
Assume first that \(\cR_{\cX}\not\equiv0\), and write its reduced
representation as
\[
\cR_{\cX}(u)=u^2\frac{P(u)}{Q(u)},
\qquad
\gcd(P,Q)=1,
\qquad
Q(0)\neq0.
\]
Under the change of Jordan chain in
Theorem~\ref{thm:restriction-transformation},
\[
\widetilde{\cR}_{\cX}(\widetilde u)
=
\frac1a\cR_{\cX}(a\widetilde u)
=
a\widetilde u^2
\frac{P(a\widetilde u)}{Q(a\widetilde u)}.
\]
Since \(a\neq0\), the substitution
\(u\mapsto a\widetilde u\) is an automorphism of the polynomial ring,
so it preserves coprimality,
\[
\gcd\bigl(P(a\widetilde u),Q(a\widetilde u)\bigr)=1,
\]
and it leaves the degree of \(Q\) unchanged,
\[
\deg Q(a\widetilde u)=\deg Q.
\]
The minimal denominator therefore has the same degree for both
Jordan chains. If \(\cR_{\cX}\equiv0\), both recurrence degrees are
zero by definition.
\end{proof}

\begin{proposition}
\label{prop:minimal-recurrence}
Let \(d_{\cX}\) be the recurrence degree of a rational distinguished
restriction. Then, beyond the finite inhomogeneous part determined by
the numerator, its Taylor coefficients satisfy a homogeneous linear
recurrence with constant coefficients of order \(d_{\cX}\).
Moreover, \(d_{\cX}\) is the smallest possible order of any eventual
homogeneous linear recurrence with constant coefficients satisfied by
the same Taylor coefficient sequence.
\end{proposition}

\begin{proof}
The first statement follows from
Theorem~\ref{thm:rational-recurrence} applied to the minimal
denominator \(Q\). For minimality, let
\[
C(u):=\sum_{n\ge0}c_nu^n
=\frac{P(u)}{Q(u)},
\qquad
\gcd(P,Q)=1,
\qquad
\deg Q=d_{\cX},
\]
and suppose that, from some index \(N\) onward, the sequence
\(c_n=g_{n+2}\) satisfies a homogeneous linear recurrence with
constant coefficients of order \(s\):
\[
\widetilde q_0c_n+\widetilde q_1c_{n-1}
+\cdots+\widetilde q_sc_{n-s}=0,
\qquad n\ge N,
\]
with
\[
\widetilde q_0\neq0,
\qquad
\widetilde q_s\neq0.
\]
Set
\[
\widetilde Q(u)
=
\widetilde q_0+\widetilde q_1u+\cdots+\widetilde q_su^s.
\]
The coefficient of \(u^n\) in
\(\widetilde Q(u)C(u)\) is precisely
\[
\widetilde q_0c_n+\widetilde q_1c_{n-1}
+\cdots+\widetilde q_sc_{n-s},
\]
which vanishes for every \(n\ge N\) by hypothesis, so all
coefficients of sufficiently high degree in the product vanish, and
therefore
\[
\widetilde Q(u)C(u)=S(u)
\]
for some polynomial \(S\). Consequently
\[
C(u)=\frac{S(u)}{\widetilde Q(u)}.
\]
After cancellation of common factors, this rational representation
has denominator of degree at most \(s\). Since \(P/Q\) is the reduced
representation of \(C\), uniqueness of the reduced denominator up to
a nonzero constant gives
\[
d_{\cX}=\deg Q\le s.
\]
No eventual homogeneous linear recurrence with constant coefficients
of order smaller than \(d_{\cX}\) can therefore be satisfied by the
same Taylor coefficient sequence.
\end{proof}

\section{Relation with smooth nilpotent normal forms}
\label{sec:normalforms}

The preceding sections concern the Taylor coefficients of the
distinguished restriction and the algebraic relations among them.
Since these coefficients are extracted from the linearly reduced
vector field, it is important to distinguish this information from
that obtained after nonlinear normal form transformations.

To make this distinction explicit, write the Taylor expansion of the
nilpotent system as
\begin{equation}
\label{eq:general-taylor-system}
\begin{aligned}
\dot u
=
v+
\sum_{j+k\geq2}
\frac{a_{jk}}{j!\,k!}u^jv^k,\qquad
\dot v
=
\sum_{j+k\geq2}
\frac{b_{jk}}{j!\,k!}u^jv^k.
\end{aligned}
\end{equation}
Then the distinguished restriction is
\begin{equation}
\label{eq:R-bjk}
\mathcal R_{\mathcal X}(u)
=
\mathcal G(u,0)
=
\sum_{j\geq2}\frac{b_{j0}}{j!}u^j.
\end{equation}
Consequently,
\begin{equation}
\label{eq:gj-bj0}
g_j=\frac{b_{j0}}{j!},
\qquad j\geq2.
\end{equation}
Thus the Taylor coefficients of the distinguished restriction are
exactly those associated with the pure monomials \(u^j\) in the
second component of the linearly reduced vector field. The
restriction contains no information about the coefficients
\(a_{jk}\) of the first component or about the mixed coefficients
\(b_{jk}\) with \(k\geq1\).

This loss of information is relevant when the system is subsequently
reduced to a smooth nilpotent normal form. The normal form procedure
uses nonlinear near--identity changes of phase coordinates, and,
depending on the chosen normalization, possibly a time
reparametrization. At each homogeneous degree the transformation is
determined through the corresponding homological equation. As a
result, coefficients that are separated in
\eqref{eq:general-taylor-system} can contribute jointly to the
coefficients that remain in the normal form; see
Kuznetsov~\cite{Kuznetsov2005}. In particular, mixed monomials that
vanish identically when \(v=0\) is imposed need not disappear from
the normal form computation.

The quadratic order already illustrates the distinction. By
Proposition~\ref{prop:g2-intrinsic},
\[
g_2
=
-\frac12 D_{q_0}\bigl(\det D\mathcal X\bigr)(0)
\]
coincides with the quadratic Bogdanov--Takens coefficient usually
denoted by \(a\). The second quadratic coefficient, however, is
\[
b
=
D_{q_0}\bigl(\operatorname{tr}D\mathcal X\bigr)(0)
=
\mathcal F_{uu}(0,0)+\mathcal G_{uv}(0,0),
\]
and therefore depends on information that is absent from
\(\mathcal R_{\mathcal X}\). Thus even at quadratic order the
restriction captures only part of the data entering the
Bogdanov--Takens normal form.

The distinction becomes more significant at higher orders. The
Taylor coefficients $g_3,g_4,\ldots$
record the complete sequence associated with the pure monomials in
the second component before nonlinear normalization. The rational
structure developed in Section~\ref{sec:rational} may therefore
impose exact relations on this entire sequence, including finite
recurrences and simultaneous cancellations. Such relations are
properties of the distinguished restriction and need not imply
corresponding cancellations among the coefficients of a smooth
nilpotent normal form, because the latter may also receive
contributions from the terms omitted by the restriction.

This point is particularly important when the distinguished
restriction becomes exactly quadratic. If
\[
\mathcal R_{\mathcal X}(u)=g_2u^2,
\]
then $b_{j0}=0,$ $j\geq3,$
so all Taylor coefficients of the restriction of degree at least
three vanish. Nevertheless, no corresponding conclusion follows for
the mixed coefficients \(b_{jk}\), \(k\geq1\), or for the
coefficients \(a_{jk}\) in the first component. Hence exact
quadraticity of the distinguished restriction is an all-orders
statement about a specific part of the linearly reduced vector field,
but it should not by itself be interpreted as an all-orders
degeneracy of the smooth nilpotent normal form.

\begin{remark}
\label{rem:pure-versus-normal}
The distinguished restriction and the smooth nilpotent normal form
therefore describe different levels of the local structure. The
restriction isolates the Taylor coefficients associated with the
pure \(u\)-monomials in the second component and retains any exact
algebraic relations among them. The normal form calculation uses the
full two-dimensional Taylor expansion and may combine pure and mixed
terms through the homological equations. Accordingly, the
distinguished restriction is best viewed as complementary to the
normal form analysis rather than as a substitute for it.
\end{remark}

\section{Applications}
\label{sec:applications}

The results of the preceding sections are now applied to four planar
models with nilpotent singularities. The examples exhibit different
behaviors of the minimal denominator of the distinguished restriction.
In the mosquito model of Huang et al. \cite{HuangRuanYuZhang2019}, the generic
recurrence degree is two and drops to one on an explicit algebraic
locus. In two modified Leslie--Gower models, the distinguished
restriction has a linear denominator, but the corresponding
cancellation mechanism behaves differently in the two cases. Finally,
a Bazykin-type model considered by Castellanos and Chan--L\'opez~\cite{CastellanosChanLopez2026}
provides a direct comparison between the intrinsic quadratic
Bogdanov--Takens coefficients and the higher order information
contained in the distinguished restriction.

\subsection{A mosquito population model: a quadratic denominator}
\label{ssec:mosquito}

We first consider the mosquito population model studied by
Huang et al.~\cite{HuangRuanYuZhang2019},
\begin{equation}
\label{eq:mosquito-model}
\begin{aligned}
\dot w
&=
\left[
\frac{w}{1+w+g}
-\bigl(\mu_1+\xi_1(w+g)\bigr)
\right]w,\\
\dot g
&=
\frac{bw}{1+w}
-\bigl(\mu_2+\xi_2(w+g)\bigr)g.
\end{aligned}
\end{equation}
All parameters are positive.

For the nilpotent cusp, the analysis in
\cite{HuangRuanYuZhang2019} is carried out under $g^*=w^*$,
together with the double-zero conditions
\[
\operatorname{tr}J(E^*)=0,
\qquad
\det J(E^*)=0.
\]
These relations determine
\(\mu_1,\mu_2,\xi_1,\xi_2\) in terms of \(b\) and \(w^*\).
For later use, write \(s=w^*>0\) and set
\[
D_M:=b(2+s)(1+2s)+s(1+s)^2.
\]
Then the equilibrium and double-zero conditions give
\begin{equation}
\label{eq:mosquito-doublezero-parameters}
\begin{aligned}
\xi_1&=
\frac{s\bigl[(1+s)^3-b(2+s)(1+2s)\bigr]}
{(1+2s)^2D_M},\\[1mm]
\xi_2&=
\frac{b\bigl[s(1+s)-b(2+s)(1+2s)\bigr]}
{s(1+s)D_M},\\[1mm]
\mu_1&=
\frac{s\bigl[b(2+s)(1+2s)(1+4s)-s(1+s)^2\bigr]}
{(1+2s)^2D_M},\\[1mm]
\mu_2&=
\frac{b\bigl[3b(2+s)(1+2s)+s(s^2-1)\bigr]}
{(1+s)D_M}.
\end{aligned}
\end{equation}
These formulas are exactly the relations obtained in
\cite{HuangRuanYuZhang2019} from the equilibrium equations together
with \(\operatorname{tr}J=\det J=0\).

The linear transformation used in
\cite{HuangRuanYuZhang2019} can be written as
\begin{equation}
\label{eq:mosquito-change}
w=w^*-Au,
\qquad
g=w^*-Bu+v,
\end{equation}
where
\begin{equation}
\label{eq:mosquito-AB}
\begin{aligned}
A&=
\frac{(w^*)^2(1+w^*)^2}
{(1+2w^*)\bigl[b(2+w^*)(1+2w^*)+w^*(1+w^*)^2\bigr]},\\
B&=
\frac{bw^*(2+w^*)}
{b(2+w^*)(1+2w^*)+w^*(1+w^*)^2}.
\end{aligned}
\end{equation}
In these coordinates the linear part is the canonical nilpotent
block.

Write system~\eqref{eq:mosquito-model} as
\[
\dot w=f(w,g),
\qquad
\dot g=h(w,g).
\]
From~\eqref{eq:mosquito-change},
\[
\dot w=-A\dot u,
\qquad
\dot g=-B\dot u+\dot v.
\]
Hence
\[
\dot u=-\frac{1}{A}f(w,g),
\]
and therefore
\begin{equation}
\label{eq:mosquito-vdot}
\dot v
=
h(w,g)-\frac{B}{A}f(w,g).
\end{equation}
Restricting to the distinguished eigendirection \(v=0\) yields the
exact scalar restriction
\begin{equation}
\label{eq:mosquito-restriction}
\mathcal R_M(u)
=
h(w^*-Au,w^*-Bu)
-\frac{B}{A}
f(w^*-Au,w^*-Bu).
\end{equation}

The following calculation aims at more than a
recurrence of order at most two: we determine the minimal denominator
of the exact restriction and identify explicitly the parameter locus
on which its degree drops.

\begin{proposition}
\label{prop:mosquito-rational}
On the double zero parameter set, the distinguished restriction of
the mosquito model has the form
\begin{equation}
\label{eq:mosquito-rational-form}
\mathcal R_M(u)
=
u^2
\frac{P_2(u)}
{\bigl(1+w^*-Au\bigr)
 \bigl(1+2w^*-(A+B)u\bigr)},
\end{equation}
where \(P_2\) is a polynomial of degree at most two.
\end{proposition}

\begin{proof}
Along \(v=0\),
\[
w=w^*-Au,
\qquad
g=w^*-Bu.
\]
Therefore the two rational denominators appearing in
system~\eqref{eq:mosquito-model} reduce to
\[
1+w=1+w^*-Au
\]
and
\[
1+w+g
=
1+2w^*-(A+B)u.
\]
All remaining terms are polynomial of degree at most two in \(u\).
Consequently, after taking a common denominator in
\eqref{eq:mosquito-restriction}, the numerator has degree at most
four.

Since the coordinates~\eqref{eq:mosquito-change} put the
linearization in canonical nilpotent form, Proposition
\ref{prop:pure-coefficients} gives
\[
\mathcal R_M(0)=\mathcal R_M'(0)=0.
\]
Hence the numerator contains the factor \(u^2\), and the remaining
factor is a polynomial of degree at most two. This proves
\eqref{eq:mosquito-rational-form}.
\end{proof}

\begin{corollary}
\label{cor:mosquito-recurrence}
Let
\[
\mathcal R_M(u)
=
\sum_{k\ge2}g_k u^k.
\]
Then, after the finitely many coefficients determined by the
quadratic numerator \(P_2\), the Taylor coefficients satisfy the
second-order recurrence
\begin{equation}
\label{eq:mosquito-recurrence}
q_0g_k+q_1g_{k-1}+q_2g_{k-2}=0
\end{equation}
for all sufficiently large \(k\), where
\[
q_0=(1+w^*)(1+2w^*),
\]
\[
q_1=
-\left[
A(1+2w^*)+(A+B)(1+w^*)
\right],
\]
and
\[
q_2=A(A+B).
\]
Thus the recurrence degree is at most two.
\end{corollary}

\begin{proposition}
\label{prop:mosquito-minimal-denominator}
Set \(s=w^*>0\) and define
\begin{equation}
\label{eq:mosquito-Psi}
\Psi(s,b)
:=
s^2(1+s)-b(2+s)(1+2s).
\end{equation}
On the biologically feasible double-zero parameter set of
Huang et al.  \cite{HuangRuanYuZhang2019}, the recurrence degree of the distinguished
restriction satisfies
\[
d_M=2\quad\text{if }\Psi(s,b)\neq0,
\]
whereas
\[
d_M=1\quad\text{if }\Psi(s,b)=0.
\]
In particular, the quadratic recurrence is generically minimal, but
its order drops by one on the algebraic locus \(\Psi=0\).
\end{proposition}

\begin{proof}
Let
\[
D_1(u)=1+s-Au,
\qquad
D_2(u)=1+2s-(A+B)u.
\]
By Proposition~\ref{prop:mosquito-rational},
\[
\mathcal R_M(u)=u^2\frac{P_2(u)}{D_1(u)D_2(u)}.
\]
Substituting the explicit double-zero relations
\eqref{eq:mosquito-doublezero-parameters} into the exact restriction
\eqref{eq:mosquito-restriction}, collecting over the denominator
\(D_1D_2\), and evaluating the resulting quadratic numerator gives
\[
P_2\!\left(\frac{1+s}{A}\right)
=
-\frac{b s^3(1+s)\,\Psi(s,b)}
{(1+2s)^2
 [\,b(2+s)(1+2s)+s(1+s)^2\,]^2},
\]
and
\[
P_2\!\left(\frac{1+2s}{A+B}\right)
=
\frac{b s^3(2+s)\,\Psi(s,b)\,[\xi_1^*(s,b)]^2}
{(1+s)(1+2s)^3
 [\,b(2+s)(1+2s)+s(1+s)^2\,]^3},
\]
where
\[
\xi_1^*(s,b)
=
(1+s)^3-b(2+s)(1+2s).
\]
On the biologically feasible parameter set one has
\(b>0\), \(s>0\), and \(\xi_1^*(s,b)>0\). Hence, if
\(\Psi(s,b)\neq0\), neither denominator factor divides \(P_2\), so
\(D_1D_2\) is already the minimal denominator and \(d_M=2\).

If \(\Psi(s,b)=0\), then
\[
\frac{1+s}{A}
=
\frac{1+2s}{A+B},
\]
so the two linear factors have the same zero and are proportional.
Substitution of
\[
b=
\frac{s^2(1+s)}
{(2+s)(1+2s)}
\]
into the exact restriction and cancellation of the common factor
gives
\[
\mathcal R_M(u)
=
\frac{s^4u^2
\left[
4s^4+4s^3+s^2-(s^3+s^2+2s+2)u
\right]}
{(1+s)(2+s)(1+2s)^5
\left[(1+2s)^2-su\right]}.
\]
The remaining numerator and denominator are coprime; indeed their
resultant with respect to \(u\), up to a nonzero constant factor, is
\[
-s^4(1+s)(2+s)(1+2s)^7(s^2+2s+2),
\]
which is nonzero for \(s>0\). Therefore the reduced denominator is
linear and \(d_M=1\).
\end{proof}

This example shows that the number of rational factors visible before
cancellation gives only an upper bound for the recurrence degree. For
the mosquito model the generic value is two, but an exact algebraic
relation makes the two denominator factors proportional and produces
a cancellation that lowers the recurrence degree to one. This
provides a mechanism, distinct from the complete cancellation in the
linear-denominator examples below, by which the Taylor coefficient
sequence of the distinguished restriction simplifies before nonlinear
normal form transformations are performed.

\begin{remark}
This example does not aim to reproduce the higher-codimension
normal form analysis of Huang et al.~\cite{HuangRuanYuZhang2019};
it shows instead that, prior to nonlinear normalization, the exact
restriction to the distinguished eigendirection already organizes its
entire Taylor coefficient sequence through a finite recurrence. In
this case the two rational denominators of the original vector field
lead generically to a minimal second-order recurrence, with an
explicit exceptional locus on which the recurrence degree drops to
one.
\end{remark}

\subsection{A modified Leslie--Gower model with Beverton--Holt-type response}
\label{ssec:zhao}
We next consider the modified Leslie--Gower model studied by
Zhao and Zhao~\cite{ZhaoZhao2026},
\begin{equation}
\label{eq:zhao-model}
\begin{aligned}
\dot x
&=
x(1-x)-\frac{\alpha x}{x+\beta}-rxy,\\
\dot y
&=
\frac{\delta y}{1+\eta y}
\left(
1-\frac{y}{x+c}
\right).
\end{aligned}
\end{equation}
A positive equilibrium has the form $E=(x_0,x_0+c).$
On the double-zero parameter set,
\begin{equation}
\label{eq:zhao-doublezero}
\beta=
\frac{1-cr-2(r+1)x_0}{r+1},
\qquad
\delta=
rx_0\bigl(1+\eta(x_0+c)\bigr),
\end{equation}
and
\[
\alpha=
\frac{\bigl(cr+(r+1)x_0-1\bigr)^2}{r+1}.
\]

Under these conditions,
\[
J(E)
=
rx_0
\begin{pmatrix}
1&-1\\
1&-1
\end{pmatrix},
\]
so that a Jordan chain is
\[
q_0=
\begin{pmatrix}
1\\1
\end{pmatrix},
\qquad
q_1=
\begin{pmatrix}
\dfrac{1}{rx_0}\\[1mm]0
\end{pmatrix}.
\]
The associated Jordan coordinates are
\begin{equation}
\label{eq:zhao-jordan}
x=x_0+u+\frac{v}{rx_0},
\qquad
y=x_0+c+u.
\end{equation}

Define
\[
\Delta:=cr+(r+1)x_0-1,
\qquad
\Theta:=cr+2(r+1)x_0-1.
\]
Then
\[
\Theta=-(r+1)\beta.
\]
Along the distinguished eigendirection \(v=0\), the exact restriction
is
\begin{equation}
\label{eq:zhao-restriction}
\mathcal R_Z(u)
=
\frac{r(r+1)^2x_0\,u^2(x_0+u)}
{\Delta-(r+1)u}.
\end{equation}

\begin{proposition}
\label{prop:zhao-application}
Assume \(\Delta\neq0\). The distinguished restriction
\eqref{eq:zhao-restriction} has recurrence degree
\[
d_Z=
\begin{cases}
1, & \Theta\neq0,\\
0, & \Theta=0.
\end{cases}
\]
Moreover, if
\[
\mathcal R_Z(u)=\sum_{k\ge2}g_k u^k,
\]
then
\[
g_2=
\frac{r(r+1)^2x_0^2}{\Delta},
\]
and, for every \(k\ge3\),
\begin{equation}
\label{eq:zhao-gk}
g_k=
\frac{r(r+1)^{k-1}x_0\,\Theta}
{\Delta^{k-1}}.
\end{equation}
Hence all Taylor coefficients of the distinguished restriction of
degree at least three share the factor \(\Theta\).
\end{proposition}

\begin{proof}
Equation~\eqref{eq:zhao-restriction} is a particular case of
Corollary~\ref{cor:linear-denominator}. The corresponding parameters
are
\[
C=r(r+1)^2x_0,
\qquad
a=x_0,
\qquad
b=1,
\]
and
\[
d=\Delta,
\qquad
e=r+1.
\]
Therefore
\[
ae+bd
=
(r+1)x_0+\Delta
=
\Theta,
\]
and~\eqref{eq:zhao-gk} follows. If \(\Theta\neq0\), the linear
denominator is coprime with the numerator and therefore the minimal
denominator has degree one. If \(\Theta=0\), the factor
\(\Delta-(r+1)u\) cancels exactly against \(x_0+u\), so the reduced
restriction is a polynomial and its recurrence degree is zero.
\end{proof}

\begin{corollary}
\label{cor:zhao-exact-cancellation}
On the algebraic locus \(\beta=0\), equivalently \(\Theta=0\),
\[
g_k=0,
\qquad k\ge3,
\]
and the distinguished restriction becomes exactly quadratic:
\begin{equation}
\label{eq:zhao-quadratic}
\mathcal R_Z(u)
=
-r(r+1)x_0u^2.
\end{equation}
\end{corollary}

Thus the Zhao--Zhao restriction exhibits the complete cancellation
described in Corollary~\ref{cor:all-orders-cancellation}: a linear
minimal denominator disappears and the recurrence degree drops from
one to zero. Since \(\Theta=-(r+1)\beta\), the cancellation locus is
equivalently described by \(\beta=0\). Its relation with the
biologically admissible parameter region should therefore be
distinguished from the purely algebraic cancellation mechanism
identified here.

\subsection{A modified Leslie--Gower model with additive Allee effect} \label{ssec:wang}
In this section  consider the modified Leslie--Gower predator--prey model
studied by Wang et al.~\cite{WangWuZou2026},
\begin{equation}
\label{eq:wang-model}
\begin{aligned}
\dot x
&=
x(1-x)-\gamma xy-\frac{\beta x}{x+\alpha},\\
\dot y
&=
\delta y
\left(
1-\frac{y}{x+\eta}
\right),
\end{aligned}
\end{equation}
where all parameters are positive.

The nilpotent positive equilibrium considered in
\cite{WangWuZou2026} is
\begin{equation}
\label{eq:wang-equilibrium}
E^*=(x_0,x_0+\eta),
\qquad
x_0=
\frac{1-\alpha-\alpha\gamma-\gamma\eta}
{2(1+\gamma)},
\end{equation}
on the parameter set
\begin{equation}
\label{eq:wang-doublezero}
\beta=(1+\gamma)(x_0+\alpha)^2,
\qquad
\delta=\gamma x_0.
\end{equation}

After translating \(E^*\) to the origin, the linear transformation
used in~\cite{WangWuZou2026} can equivalently be written as
\begin{equation}
\label{eq:wang-change}
x=x_0+u+\frac{v}{\gamma x_0},
\qquad
y=x_0+\eta+u.
\end{equation}
In these coordinates the linearization is the canonical nilpotent
block.

Along the distinguished eigendirection \(v=0\),
\begin{equation}
\label{eq:wang-line}
x=x_0+u,
\qquad
y=x_0+\eta+u=x+\eta.
\end{equation}
Hence the second component of~\eqref{eq:wang-model} vanishes
identically,
\[
\delta y
\left(
1-\frac{y}{x+\eta}
\right)
\equiv0.
\]
Moreover,
\[
v=\gamma x_0
\bigl[(x-x_0)-(y-x_0-\eta)\bigr],
\]
and therefore
\[
\dot v=\gamma x_0(\dot x-\dot y).
\]
It follows that
\begin{equation}
\label{eq:wang-restriction-pre}
\mathcal R_W(u)
=
\gamma x_0 H(x_0+u),
\end{equation}
where
\[
H(x)
=
x\left[
1-\gamma\eta-(1+\gamma)x
-\frac{\beta}{x+\alpha}
\right].
\]

\begin{proposition}
\label{prop:wang-restriction}
On the double-zero parameter set~\eqref{eq:wang-doublezero}, the
distinguished restriction is
\begin{equation}
\label{eq:wang-restriction}
\mathcal R_W(u)
=
-\gamma(1+\gamma)x_0
\frac{u^2(x_0+u)}
{x_0+\alpha+u}.
\end{equation}
\end{proposition}

\begin{proof}
The double-zero relations imply
\[
\beta=(1+\gamma)(x_0+\alpha)^2
\]
and
\[
1-\gamma\eta
=
(1+\gamma)(2x_0+\alpha).
\]
Substitution into \(H(x_0+u)\) gives
\[
H(x_0+u)
=
-(1+\gamma)
\frac{u^2(x_0+u)}
{x_0+\alpha+u}.
\]
Equation~\eqref{eq:wang-restriction} follows from
\eqref{eq:wang-restriction-pre}.
\end{proof}

\begin{corollary}
\label{cor:wang-coefficients}
In the admissible parameter region \(\alpha>0\), the Wang--Wu--Zou
restriction has recurrence degree $d_W=1.$
If
\[
\mathcal R_W(u)=\sum_{k\ge2}g_k u^k,
\]
then
\[
g_2
=
-\frac{\gamma(1+\gamma)x_0^2}
{x_0+\alpha},
\]
and, for every \(k\ge3\),
\begin{equation}
\label{eq:wang-gk}
g_k=(-1)^{k-2}
\frac{\gamma(1+\gamma)x_0\alpha}
{(x_0+\alpha)^{k-1}}.
\end{equation}
\end{corollary}

\begin{proof}
Equation~\eqref{eq:wang-restriction} is a particular case of
Corollary~\ref{cor:linear-denominator}, with
\[
C=-\gamma(1+\gamma)x_0,
\qquad
a=x_0,
\qquad
b=1,
\]
and
\[
d=x_0+\alpha,
\qquad
e=-1.
\]
Thus
\[
\Sigma=ae+bd
=
-x_0+(x_0+\alpha)
=
\alpha.
\]
Since \(\alpha>0\), the numerator and the linear denominator cannot
have the cancellation characterized in
Corollary~\ref{cor:all-orders-cancellation}. Hence the denominator
remains minimal of degree one, and the coefficient formula follows.
\end{proof}

\begin{remark}
\label{rem:wang-no-cancellation}
For the Wang et al.  model the factor controlling the possible
linear-denominator cancellation is \(\Sigma=\alpha\). Since
\(\alpha>0\) in the admissible parameter region, the Taylor
coefficients of the distinguished restriction of degree at least
three cannot vanish simultaneously. Thus the higher-codimension
nilpotent degeneracies identified in~\cite{WangWuZou2026} are not
produced by an all orders cancellation in the distinguished
restriction. This illustrates the distinction between the exact
restriction and the coefficients obtained after nonlinear normal form
transformations.
\end{remark}

\subsection{A Bazykin model revisited: comparison with the intrinsic BT coefficients}
\label{ssec:bazykin-comparison}
To make the relation with the intrinsic approach of Castellanos and
Chan--L\'opez~\cite{CastellanosChanLopez2026} explicit, we revisit one
of the Bazykin-type systems analyzed in that work. The model, its
Bogdanov--Takens curve, and the quadratic coefficients quoted below
are due to Castellanos and Chan--L\'opez \cite{CastellanosChanLopez2026}. Here we take a different
route: starting from their Bogdanov--Takens configuration, we
apply the distinguished restriction developed in the present paper
to determine its exact rational structure and recurrence degree.

Consider the mutual-interference model
\begin{equation}
\label{eq:bazykin-mutual}
\dot x=x-\frac{xy}{(1+\alpha x)(1+\beta y)},
\qquad
\dot y=-\gamma y+\frac{xy}{(1+\alpha x)(1+\beta y)},
\end{equation}
with \(\alpha,\beta,\gamma>0\). Castellanos and Chan--L\'opez show
that its Bogdanov--Takens set is parametrized by
\begin{equation}
\label{eq:bazykin-bt-curve}
x_0=(1+\gamma)^2,
\qquad
y_0=\frac{(1+\gamma)^2}{\gamma},
\qquad
\alpha=\frac{\gamma}{(1+\gamma)^2},
\qquad
\beta=\frac{1}{(1+\gamma)^2},
\end{equation}
and, for their normalization of the kernel direction, the intrinsic
quadratic coefficients are
\begin{equation}
\label{eq:bazykin-ab}
a(\gamma)=\frac{\gamma^3}{(1+\gamma)^4},
\qquad
b(\gamma)=
-\frac{(\gamma-1)\gamma^2}{(1+\gamma)^4}.
\end{equation}
In particular, \(b=0\) only at \(\gamma=1\), whereas \(a>0\) for all
\(\gamma>0\).

\begin{proposition}
\label{prop:bazykin-restriction}
Along the Bogdanov--Takens curve~\eqref{eq:bazykin-bt-curve}, choose
\[
q_0=(\gamma,1)^\top,
\qquad
q_1=(1+\gamma,0)^\top.
\]
Then the distinguished restriction of~\eqref{eq:bazykin-mutual} is
\begin{equation}
\label{eq:bazykin-restriction}
\mathcal R_B(u)=
\frac{\gamma^3u^2\bigl((1+\gamma)^2+\gamma u\bigr)}
{\bigl((1+\gamma)^3+\gamma u\bigr)
 \bigl((1+\gamma)^3+\gamma^2u\bigr)}.
\end{equation}
For every \(\gamma>0\) the numerator is coprime with the denominator,
and hence $d_B=2.$
Moreover,
\[
g_2=
\frac{\gamma^3}{(1+\gamma)^4}
=
a(\gamma),
\]
in agreement with the intrinsic coefficient in
\eqref{eq:bazykin-ab}.
\end{proposition}

\begin{proof}
On~\eqref{eq:bazykin-bt-curve}, direct differentiation gives
\[
D\mathcal X(x_0,y_0)
=
\frac1{1+\gamma}
\begin{pmatrix}
\gamma&-\gamma^2\\
1&-\gamma
\end{pmatrix}.
\]
Thus
\[
D\mathcal X(x_0,y_0)q_0=0,
\qquad
D\mathcal X(x_0,y_0)q_1=q_0.
\]
With \(P=(q_0\ q_1)\), the second row of \(P^{-1}\) is
\[
\left(
\frac1{1+\gamma},
-\frac{\gamma}{1+\gamma}
\right).
\]
Substituting
\[
(x,y)=(x_0,y_0)+u q_0
\]
into the vector field and projecting with this row gives
\eqref{eq:bazykin-restriction} after simplification.

The nontrivial numerator factor has the zero
\[
u_N=-\frac{(1+\gamma)^2}{\gamma},
\]
whereas the two denominator factors have zeros
\[
u_1=-\frac{(1+\gamma)^3}{\gamma},
\qquad
u_2=-\frac{(1+\gamma)^3}{\gamma^2}.
\]
For \(\gamma>0\), the equality \(u_N=u_1\) would imply
\(1=1+\gamma\), while \(u_N=u_2\) would imply
\(\gamma=1+\gamma\). Both are impossible. Hence no denominator
factor cancels with the numerator, and the minimal denominator has
degree two. Notice that the two denominator factors themselves
coincide when \(\gamma=1\); this produces a repeated factor but does
not reduce the degree of the minimal denominator.

Finally, evaluation of the rational factor at \(u=0\) gives
\[
g_2
=
\gamma^3
\frac{(1+\gamma)^2}{(1+\gamma)^6}
=
\frac{\gamma^3}{(1+\gamma)^4},
\]
which agrees with~\eqref{eq:bazykin-ab}.
\end{proof}

\begin{remark}
\label{rem:bazykin-complementarity}
This example makes the complementarity of the two viewpoints
explicit. The coefficient \(g_2\) reproduces the intrinsic
Bogdanov--Takens coefficient \(a\) of Castellanos and
Chan--L\'opez \cite{CastellanosChanLopez2026}, while the exact restriction additionally determines
the complete Taylor coefficient sequence through a second-order
recurrence.

The point \(\gamma=1\) is particularly instructive. There, $b(1)=0,$
whereas
\[
\mathcal R_B(u)
=
\frac{u^2(u+4)}{(u+8)^2},
\qquad
d_B=2.
\]
Thus the two denominator factors coalesce into a repeated factor, but
no cancellation with the numerator occurs and the recurrence degree
does not drop. In particular, the degeneracy \(b=0\) of the second
quadratic Bogdanov--Takens coefficient does not imply a cancellation
in the distinguished restriction. This is consistent with
Section~\ref{sec:normalforms}: the coefficient \(b\) depends on mixed
Taylor information that is not contained in
\(\mathcal R_{\mathcal X}\) alone.
\end{remark}

\subsection{Comparison of the examples}
\label{ssec:comparison}

The preceding examples exhibit different ways in which the minimal
denominator of the distinguished restriction can behave.

In the mosquito model~\cite{HuangRuanYuZhang2019}, the distinguished
restriction has generically a minimal quadratic denominator, so its
Taylor coefficients satisfy a minimal second-order recurrence. On the
algebraic locus \(\Psi=0\), the two denominator factors become
proportional and a common factor cancels. The recurrence degree
therefore drops from two to one.

In the modified Leslie--Gower model~\cite{ZhaoZhao2026}, the minimal
denominator is linear. The Taylor coefficients of degree at least
three share the factor \(\Theta\), and on the algebraic locus
\(\Theta=0\) the numerator and denominator cancel exactly. The
restriction then becomes quadratic and the recurrence degree drops
from one to zero.

The second modified Leslie--Gower model~\cite{WangWuZou2026} also has
a linear minimal denominator, but the factor controlling the possible
cancellation is \(\Sigma=\alpha\). Since \(\alpha>0\) in the
admissible parameter region, no such cancellation occurs and the
recurrence degree remains equal to one.

The Bazykin model considered in
\cite{CastellanosChanLopez2026} exhibits a different phenomenon. Its
minimal denominator has degree two for every \(\gamma>0\). At
\(\gamma=1\), where the second intrinsic Bogdanov--Takens coefficient
satisfies \(b=0\), the two denominator factors coalesce into a repeated
factor, but no factor cancels with the numerator. Consequently,
\(d_B=2\) also at this degenerate point.

Taken together, the four examples distinguish several algebraic
mechanisms: cancellation after coalescence of denominator factors,
complete cancellation of a linear denominator, exclusion of
cancellation in an admissible parameter region, and coalescence
without cancellation. Thus the recurrence degree is determined by the
reduced exact restriction, rather than by the number of rational
factors visible before reduction or by the codimension of the
corresponding nilpotent normal form.

\section{Concluding remarks}
\label{sec:conclusions}

We have studied a one-dimensional object naturally associated with a
planar vector field whose linearization is a nonzero nilpotent matrix
of index two. The intrinsic map on the kernel of the linearization
admits, after choosing a generator of that kernel, a scalar
representation \(\mathcal R_{\mathcal X}\). In Jordan coordinates this
representation is precisely
\[
\mathcal R_{\mathcal X}(u)=\mathcal G(u,0),
\]
so its Taylor coefficients coincide with the coefficients associated
with the pure \(u\)-monomials in the second component of the linearly
reduced vector field.

The dependence of this scalar representation on the choice of Jordan
chain is completely described by the scaling law
\[
\widetilde{\mathcal R}_{\mathcal X}(\widetilde u)
=
\frac1a\mathcal R_{\mathcal X}(a\widetilde u).
\]
Accordingly, the individual Taylor coefficients depend on the
normalization of the kernel direction, whereas their vanishing pattern
and the restriction order do not. In the rational case, the degree of
the minimal denominator is likewise independent of the admissible
Jordan chain.

The quadratic Taylor coefficient occupies a special position. It
coincides with the intrinsic quadratic Bogdanov--Takens coefficient
characterized in~\cite{CastellanosChanLopez2026}. The remaining Taylor
coefficients, however, should be interpreted at the level at which
they are defined: they describe the linearly reduced vector field
modulo admissible changes of Jordan chain and are not claimed to be
invariants under arbitrary nonlinear phase-coordinate transformations.

For rational distinguished restrictions, the minimal denominator
provides a finite algebraic description of the complete Taylor
coefficient sequence. Its degree is the minimal order of an eventual
homogeneous linear recurrence with constant coefficients satisfied by
that sequence. Exact cancellations between numerator and denominator
therefore appear as changes in recurrence degree, while the
linear-denominator case gives an explicit criterion for simultaneous
vanishing of all Taylor coefficients of degree at least three.

The applications show that these recurrence properties capture
algebraic structure that is already present before nonlinear
normalization. They need not parallel the higher-codimension
classification obtained from a smooth nilpotent normal form, since
mixed terms that disappear under the distinguished restriction may
still contribute to normal form coefficients. The distinguished
restriction is therefore best viewed as complementary to normal form
theory: it isolates exact relations in the linearly reduced vector
field, including finite recurrence laws and all-orders cancellations,
without replacing the information contained in the full
two-dimensional normal form calculation.

\bigskip

\section*{Acknowledgments}
The authors thank Eduardo Chan--L\'opez for his careful reading of an
earlier version of the manuscript and for pointing out an intrinsic
formulation of the distinguished restriction on the nilpotent
eigendirection. His observation helped us clarify the role of the
Jordan chain and the transformation law of the scalar restriction.
\bigskip

\paragraph{Declaration on generative AI}
AI tools were used as an aid in organizing and refining the
presentation of the material, as well as for language editing and
assistance with bibliographic formatting and verification. All
references were independently verified by the authors. All
mathematical content, results, proofs, and computations (carried out
in Mathematica) are the authors' own, and the authors take full
responsibility for the final content of the manuscript.

\bigskip

\noindent{\bf Funding}  The second author is partially supported by the 2026 Special Program for Teaching and Research Project Funding at CBI UAM-Iztapalapa.
\medskip

\noindent{\bf Author Contributions} 
R. Albarran-García: Methodology, Formal analysis, Investigation,
Writing – original draft. M. Alvarez-Ramírez: Conceptualization, Methodology, Formal analysis,
Writing – review and editing, Supervision. M. P. García-Rivera: Formal analysis, Investigation,
Writing – review and editing.
\medskip

\noindent{\bf Data availability} The data that supports the findings of this study are available within the article.
\medskip

\noindent\textbf{Competing Interests}
The authors declare that they have no competing interests.


\end{document}